\documentclass[12pt]{amsart}
\usepackage[usenames]{color}
\usepackage{times}
\usepackage{amsfonts}
\usepackage{amsthm}
\usepackage{amsmath}
\usepackage{psfrag}
\usepackage{times}
\usepackage{latexsym,color}
\usepackage{amssymb}
\usepackage{graphicx}
\usepackage{epsfig}
\usepackage{gastex}
\advance \topmargin by -\headheight
\advance \topmargin by -\headsep
\evensidemargin \oddsidemargin
\newcommand{\p}{\rho}
\newcommand{\Fin}{\mbox{Fin}(\nats)}
\newcommand{\Zcal}{\mathcal{Z}}
\newcommand{\F}{\mathcal{F}}

\newcommand{\larr}{\left( \begin{array}{c}}
\newcommand{\rarr}{\end{array} \right) }

\newcommand{\lsqarr}{\left[ \begin{array}{c}}
\newcommand{\rsqarr}{\end{array} \right]}

\makeatletter
\def\Ddots{\mathinner{\mkern1mu\raise\p@
\vbox{\kern7\p@\hbox{.}}\mkern2mu
\raise4\p@\hbox{.}\mkern2mu\raise7\p@\hbox{.}\mkern1mu}}
\makeatother

\newcommand{\reals}{{\mathbb R}}

\newcommand{\nats}{{\mathbb N}}
\newcommand{\ints}{{\mathbb Z}}

\newtheorem{thm}{Theorem}

\newtheorem{theorem}{Theorem}[section]

\newtheorem{lemma}[theorem]{Lemma}
\newtheorem{corollary}[theorem]{Corollary}
\newtheorem{proposition}[theorem]{Proposition}

\newtheorem{remark}[theorem]{Remark}

\newtheorem{definition}[theorem]{Definition}
\theoremstyle{definition}

\begin{document}

\title[Central sets generated by uniformly recurrent words]{Central sets generated by uniformly recurrent words}

\author[M. Bucci]{Michelangelo Bucci}
\address{Department of Mathematics\\
FUNDIM\\
University of Turku\\
FIN-20014 Turku, Finland.}
\email{michelangelo.bucci@utu.fi \\ micbucci@unina.it}

\author[S. Puzynina]{Svetlana Puzynina}
\address{Department of Mathematics\\
FUNDIM\\
University of Turku\\
FIN-20014 Turku, Finland. Also, Sobolev Institute of Mathematics \\4 Acad. Koptyug avenue \\630090 Novosibirsk Russia}
\email{svetlana.puzyina@utu.fi}

\author[L.Q. Zamboni]{Luca Q. Zamboni}
\address{Universit\'e de Lyon\\
Universit\'e Lyon 1\\
CNRS UMR 5208\\
Institut Camille Jordan\\
43 boulevard du 11 novembre 1918\\
F69622 Villeurbanne Cedex, France. Also,
Department of Mathematics\\
FUNDIM\\
University of Turku\\
FIN-20014 Turku, Finland.}
\email{zamboni@math.univ-lyon1.fr\\
luca.zamboni@utu.fi}

\date{April 30, 2012}

\maketitle

\begin{abstract}
A subset $A$ of $\nats$ is called an  IP-set if $A$ contains all finite sums of distinct terms of some infinite sequence $(x_n)_{n\in \nats} $ of natural numbers. Central sets, first introduced by  Furstenberg using notions from topological dynamics, constitute a special class of IP-sets possessing rich combinatorial properties: Each central set contains arbitrarily long arithmetic progressions, and solutions to all partition regular systems of homogeneous
linear equations.    In this paper we investigate central sets in the framework of combinatorics on words. Using various families of uniformly recurrent words, including  Sturmian words, the Thue-Morse word and fixed points of weak mixing substitutions, we generate an assortment of central sets which reflect the rich combinatorial structure of the underlying words. The results in this paper rely on interactions between different areas of mathematics, some of which had not previously been directly linked. They include the general theory of combinatorics on words,  abstract numeration systems, and the beautiful theory, developed by Hindman, Strauss and others, linking  IP-sets and central sets to the algebraic/topological properties of the  Stone-\v Cech compactification of  $\nats .$\end{abstract}

\section{Introduction}

Let $\nats =\{0,1,2,3,\ldots\}$ denote the set of natural numbers, and $\Fin$ the set of all non-empty finite subsets of $\nats.$
\begin{definition}
A subset $A$ of $\nats$ is called an {\it IP-set} if $A$ contains $\{\sum _{n\in F}x_n \,| \,F\in \Fin\}$ for some infinite sequence of natural numbers $x_0<x_1<x_2 \cdots .$ A subset $A\subseteq \nats$  is called an {\it IP$^*$-set} if $A\cap B\neq \emptyset $ for every IP-set $B\subseteq \nats.$
\end{definition}
By a celebrated result of N. Hindman \cite{H}, given any finite partition of $\nats,$ at least one element of the partition is an IP-set. It follows from Hindman's theorem that every IP$^*$-set is an IP-set, but the converse is in general not true.  In fact, more generally Hindman shows that given any finite partition of an IP-set, at least one element of the partition is again an IP-set. In other words the property of being an IP-set is {\it partition regular}, i.e., cannot be destroyed via a finite partitioning. Other examples of partition regularity are given by the pigeonhole principle, sets having positive upper density, and sets having arbitrarily long arithmetic progressions (Van der Waerden's theorem). In \cite{F}, Furstenberg introduced a special class of IP-sets, called central sets, having a substantial combinatorial structure. The property of being central is also partition regular. Central sets were originally defined in terms of  topological dynamics:

\begin{definition}\label{Cen2} A subset $A\subset \nats$ is called {\it central} if there exists a compact metric space $(X,d)$ and a continuous map $T: X\rightarrow X,$ points $x,y \in X$ and a neighborhood $U$ of $y$ such that
\begin{itemize}
\item $y$ is a uniformly recurrent point in $X,$
\item $x$ and $y$ are proximal,
\item $A=\{n \in \nats\,|\, T^n(x)\in U\}.$
\end{itemize}
 We say $A\subset \nats$ is {\it central}$^*$ if  $A\cap B\neq \emptyset $ for every central set $B\subseteq \nats.$
\end{definition}

\noindent Recall that $x$ is said to be {\it uniformly recurrent} in $X$ if for every neighborhood $V$ of $x$ the set
$\{n\,|\, T^n(x)\in V\}$ is syndetic, i.e., of bounded gap. Two points $x,y\in X$ are said to be {\it proximal} if for every $\epsilon >0$ there exists $n\in \nats$ such that $d(T^n(x),T^n(y))<\epsilon.$   We remark that from the above definition, it is not at all evident that central sets are IP-sets. We later give an alternative definition (see Definition~\ref{Cen1}) which makes this point clear. The equivalence between the two definitions is due to Bergelson and Hindman \cite{BH}.\\

The question of determining whether a given subset $A\subseteq \nats$ is an IP-set or a central set is typically quite difficult, even if for every $A,$ either $A$ or its complement is an IP-set (resp. central set). It turns out that in each case this question may be reformulated in terms of whether or not the set $A$ belongs to a certain class of ultrafilters on $\nats$ (see Theorem 5.12 in \cite{HS} in the case of IP-sets and \cite{BH} in the case of central sets). But the question of belonging or not to a given (non-principal) ultrafilter is  generally equally mysterious.
An equivalent  word combinatorial reformulation of this question is as follows: Given a binary word $\omega =\omega_0\omega_1\omega_2\ldots \in \{0,1\}^\infty,$ put $\omega \big|_0=\{n\in \nats \,| \,\omega_n=0\}$ and $\omega \big|_1=\{n\in \nats \,| \,\omega_n=1\}.$  The question is then to determine whether the set $\omega \big|_0$ or $\omega \big|_1$ is an IP-set or central set. Of course in general, this reformulation is as difficult as the original question. However, should the word $\omega$ be characterized by some rich combinatorial properties, or be  generated by some ``simple" combinatorial or geometric algorithm (such as a substitution rule, a finite state automaton,  a Toeplitz rule...) or arise as a natural coding of a reasonably simple symbolic dynamical system, then the underlying rigid combinatorial structure of the word may provide insight to our previous question. Furthermore,  such families of  words may be used to obtain simple constructions of central sets having additional nice properties inherited from the rich underlying combinatorial structure. One of our objectives here is to illustrate this latter point.\\

Let $\mathcal A$ denote a finite non-empty set (called the alphabet) and $\omega =\omega_0\omega_1\omega_2\ldots\in {\mathcal A}^{\nats}.$
For each finite word $u$ on the alphabet ${\mathcal A}$ we set
\[\omega\big|_{u}=\{ n\in \nats \,|\, \omega_n\omega_{n+1}\ldots \omega_{n+|u|-1}=u\}.\]
In other words, $\omega\big|_u$ denotes the set of all occurrences of $u$ in $\omega.$

\noindent In this paper we investigate partitions of $\nats$ by sets of the form $\omega\big|_{u}$ defined by a uniformly recurrent word $\omega.$  Our goal is to study these partitions  in the framework of IP-sets and central sets.   We begin by showing that in this framework IP-sets and central sets are one and the same:

\begin{thm} Let $\omega \in {\mathcal A}^{\nats}$ be uniformly recurrent. Then the set
$\omega\big|_u$ is an IP-set if and only if it is a central set.
\end{thm}

\noindent This allows us to simultaneously state our results in terms of IP-sets and central sets.\\

We begin by considering the {\it simplest} aperiodic infinite words, namely  Sturmian words.
Sturmian words are infinite words over a binary alphabet having exactly $n+1$ factors of length
$n$ for each $n \geq 0.$ Their origin can be traced back to the astronomer J. Bernoulli  III in 1772.  A fundamental result due to Morse and Hedlund \cite{MorHed1940} states that each aperiodic (meaning non-ultimately periodic) infinite word must contain at least $n+1$ factors of each length $n\geq 0.$  Thus Sturmian words are those aperiodic words of lowest factor complexity.  They arise naturally in many different areas of mathematics including combinatorics, algebra, number theory, ergodic theory, dynamical systems and differential equations. Sturmian words are also of great importance in theoretical physics and in theoretical computer science and are used in
computer graphics as digital approximation of straight lines. \\

The next two theorems give a complete characterization of those factors 
$u$  of a Sturmian word $\omega\in\{0,1\}^\nats$ for which $\omega \big|_u$ is an IP-set (respectively central set). 
First,  a Sturmian word $\omega $ is called {\it singular} if $T^n(\omega)= \tilde \omega$ for some $n\geq 1,$ where $T$ denotes the shift map and $\tilde \omega$ denotes the characteristic Sturmian word in the shift orbit closure of $\omega$ (see \S\ref{sss} for the definition of a characteristic Sturmian word).  Otherwise it is said to be  {\it nonsingular.}

 \begin{thm}\label{sturm1} Let $\omega \in \Omega$ be a nonsingular Sturmian word, and $u$ a factor of $\omega.$  Then $\omega \big|_u$ is an IP-set (resp. central set)  if and only if $u$ is a prefix of $\omega.$  Hence for every prefix $v$ of $\omega$
 and $n\in \omega \big|_v$ the set $\omega \big|_v-n$ is an IP$^*$-set (resp. central$^*$ set).\end{thm}
 
 \begin{thm}\label{sturm2} Let $\omega \in \Omega$ be a Sturmian word such that $T^{n_0}(\omega)=\tilde \omega$ with $n_0\geq1.$  Then $\omega\big|_u$ is an IP-set (or central set) if and only if either $u$ is a prefix of $\omega$ or a prefix of $\omega'$ where  $\omega' $ is the unique other element of $\Omega$ with $T^{n_0}(\omega')=\tilde \omega.$
\end{thm}

\noindent Some (but not all) of the results on Sturmian partitions extend to the class of  Arnoux-Rauzy words, which may be regarded as natural combinatorial extensions of Sturmian words to larger alphabets \cite{ArRa}. \\

\noindent Using $\omega$-bonacci and the iterated palindromic closure  operator, we construct infinite partitions of $\nats$ into central sets having special translation invariant properties.

We also consider partitions defined by words generated by substitution rules. For instance, by considering partitions of $\nats$ defined by words generated by the generalized Thue-Morse substitution to an alphabet of size $r\geq 2,$ we show that

\begin{thm}\label{T3} For each pair of positive integers $r$ and $N$ there exists a partition of \[\nats=A_1\cup A_2\cup \cdots \cup A_r\] such that
\begin{itemize}
\item $A_i-n$ is a central set for each $1\leq i\leq r$ and $1\leq n\leq N.$
\item For each $n>N,$ exactly one of the sets $\{A_1-n, A_2-n, \ldots ,A_r-n\}$ is a central set.
\end{itemize}
\end{thm}

\noindent The second assertion of Theorem~\ref{T3} relies on the fact that each fixed point of the generalized Thue-Morse substitution is distal. 

By considering partitions defined by words generating minimal subshifts  which are topologically weak mixing (for example the subshift generated by the substitution $0\mapsto 001$ and $1\mapsto 11001)$ we prove that

\begin{thm}\label{T4} For each  positive integer $r$ there exists a partition of $\nats=A_1\cup A_2\cup \cdots \cup A_r$ such that for each $1\leq i\leq r$ and $n\geq 0,$ the set $A_i-n$ is a central set.
\end{thm}

The results in this paper rely on various interactions between combinatorics on words, topological dynamics and the algebraic and topological properties of the Stone-\v Cech compactification $\beta \nats .$
We regard $\beta \nats$ as the collection of all ultrafilters on $\nats.$ An ultrafilter may be thought of as a $\{0,1\}$-valued finitely additive probability measure defined on all subsets of $\nats.$ This notion of measure induces a notion of convergence ($p$-$\lim_n)$ for sequences indexed by $\nats,$ which we regard as a mapping $p^*$ from words to words. This key notion of convergence allows us to apply ideas from combinatorics on words in the framework of ultrafilters. \\

\paragraph{\bf Acknowledgements}
The authors would like to thank V. Bergelson and Y. Son for many insightful e-mail exchanges and in particular for pointing out to us the key feature used in the proof of Theorem~\ref{T4} relating topologically weak mixing with proximality. We are also extremely grateful  to N. Hindman for his comments and suggestions on a preliminary version of this paper.  The third author is partially supported by a grant from the Academy of Finland.

\section{Words and substitutions}
In this section we give a brief summary of some of the basic background
in combinatorics on words.

\subsection{Words \& subshifts}
Given a finite non-empty set ${\mathcal A}$ (called the {\it alphabet}), we denote by ${\mathcal A}^*,$  ${\mathcal A}^\nats$  and ${\mathcal A}^\ints$ respectively the set of finite words,  the set of (right) infinite words, and the set of bi-infinite words over the alphabet ${\mathcal A}$. Given a finite word $u =a_1a_2\ldots a_n$ with $n \geq 1$ and $a_i \in {\mathcal A},$ we denote the length $n$ of $u$ by $|u|.$ The  \textit{empty word} will be denoted by $\varepsilon$ and we set $|\varepsilon|=0.$ We put $ {\mathcal A}^+= {\mathcal A}^*-\{\varepsilon\}.$ For each $a\in {\mathcal A},$ we let $|u|_a$  denote the number of occurrences of the letter $a$ in $u.$

Given an infinite word
$\omega \in {\mathcal A}^\nats,$ a word $u\in {\mathcal A}^+$ is called a  {\it factor} of $\omega$
if  $u=\omega_{i}\omega_{i+1}\cdots \omega_{i+n}$ for some natural numbers $i$ and $n.$
We denote by ${\mathcal F}_{\omega}(n)$ the set of all factors of $\omega$ of length $n,$ and set
\[{\mathcal F}_{\omega} =\bigcup _{n\in \nats} {\mathcal F}_{\omega}(n).\]
A factor $u$ of $\omega$ is called {\it right special} if both $ua$ and $ub$ are factors of $\omega$ for some pair of distinct letters $a,b \in {\mathcal A}.$ Similarly $u$ is called {\it left special} if both $au$ and $bu$ are factors of $\omega$ for some pair of distinct letters $a,b \in {\mathcal A}.$ The factor $u$ is called $bispecial$ if it is both right special and left special.
For each factor $u\in {\mathcal F}_{\omega}$  set
\[\omega\big|_{u}=\{ n\in \nats \,|\, \omega_n\omega_{n+1}\ldots \omega_{n+|u|-1}=u\}.\]
We say $\omega$ is {\it recurrent} if for every $u\in {\mathcal F}_{\omega}$  the set $\omega\big|_u$ is infinite.
We say $\omega$ is {\it uniformly recurrent} if for every $u\in {\mathcal F}_{\omega}$  the set $\omega\big|_u$ is syndedic, i.e., of bounded gap.

We endow ${\mathcal A}^\nats$ with the topology generated by the metric
\[d(x, y)=\frac 1{2^n}\,\,\mbox{where} \,\, n=\inf\{k :x_k\neq y_k\}\] %replaced min by inf
whenever $x=(x_n)_{n\in \nats}$ and $y=(y_n)_{n\in \nats}$ are two elements of ${\mathcal A}^\nats.$ Let $T:{\mathcal A}^\nats \rightarrow {\mathcal A}^\nats$ denote the {\it shift} transformation defined by $T: (x_n)_{n\in \nats}\mapsto (x_{n+1})_{n\in \nats}.$ By a {\it subshift} on ${\mathcal A}$ we mean a pair $(X,T)$ where $X$ is a closed and $T$-invariant subset of ${\mathcal A}^\nats.$ A subshift $(X,T)$ is said to be {\it minimal}
whenever $X$ and the empty set are the only $T$-invariant closed subsets of $X.$ To each $\omega \in {\mathcal A}^\nats$ is associated the subshift $(X,T)$ where $X$ is the shift orbit closure of $\omega.$ If $\omega$ is uniformly recurrent, then the associated subshift $(X,T)$ is minimal.
Thus any two words $x$ and $y$ in $X $ have exactly the same set of factors, i.e., ${\mathcal F}_x={\mathcal F}_y.$ In this case we denote by ${\mathcal F}_{X}$ the set of factors of any word $x\in X.$

Two points $x ,y$ in $X$ are said to be {\it proximal} if and only if for each $N>0$ there exists $n\in \nats$ such that  \[x_nx_{n+1}\ldots x_{n+N}= y_ny_{n+1}\ldots y_{n+N}.\]
Two points $x,y\in X$ are said to be {\it regionally proximal} if for every prefix $u$ of $x$ and $v$ of $y,$ there exist points $x',y'\in X$ with $x'$ beginning in $u$ and $y'$ beginning in $v$ and with $x'$ proximal to $y'.$ Clearly if two points in $X$ are proximal, then they are regionally proximal.
A point $x\in X$ is called {\it distal} if the only point in $ X$ proximal to $x$ is $x$ itself.  A minimal subshift $(X,T)$
is said to be {\it topologically mixing} if for every any pair of factors $u,v \in {\mathcal F}_X$ there exists a positive integer $N$ such that for each $n\geq N,$ there exists a block of the form $uWv \in  {\mathcal F}_X$ with $|W|=n.$
A minimal subshift $(X,T)$
is said to be {\it topologically weak mixing} if for every pair of factors $u,v \in {\mathcal F}_X$
the set
\[\{n \in \nats\,|\, u{\mathcal A}^nv \cap {\mathcal F}_X \neq \emptyset\}\]
is thick, i.e., for every positive integer $N,$ the set contains $N$ consecutive positive integers.

Many of the words and subshifts  considered in this paper are generated by substitutions.
A {\it substitution} $\tau $ on an alphabet $ {\mathcal A}$ is
a mapping $\tau : {\mathcal A}\rightarrow  {\mathcal A}^+.$
The mapping
$\tau $  extends by concatenation to maps (also
denoted $\tau )$
$ {\mathcal A} ^*\rightarrow  {\mathcal A} ^*$ and $ {\mathcal A}^{\nats}\rightarrow  {\mathcal A}^{\nats}.$

Let $\tau $ be a primitive substitution on $ {\mathcal A}.$
A word $\omega \in  {\mathcal A}^{\nats}$ is called a {\it fixed point} of $\tau$
if $\tau (\omega)=\omega,$ and is called a {\it periodic point} if $\tau ^m(\omega)=\omega$ for
some $m>0.$
Although $\tau $ may fail to have a fixed point, it has at least one periodic point.
Associated to $\tau $ is the topological dynamical system
$(X,T),$ where $X$ is the shift orbit closure of a periodic point $\omega$ of $\tau.$
The primitivity of $\tau $ implies that $(X,T)$ is
independent of the choice of periodic point and is  minimal.

\subsection{Sturmian words \& generalizations}\label{sss}
Let $\omega \in {\mathcal A}^\nats$ and  set \[\p_{\omega}(n)=\mbox{Card}({\mathcal F}_{\omega}(n)).\] The function $\p_{\omega}:\nats \rightarrow \nats$ is called the {\it factor complexity function} of $\omega.$
Given a minimal subshift $(X,T)$ on $A,$ we have
${\mathcal F}_{\omega}(n)={\mathcal F}_{\omega'}(n)$ for
all $\omega, \omega '\in X$ and $n\in \nats.$ Thus we can define the factor complexity $\p_{(X,T)}(n)$ of a minimal subshift $(X,T)$ by
\[ \p_{(X,T)}(n)=\p_{\omega}(n)\]
for any $\omega \in X.$

A word $\omega\in {\mathcal A}^\nats$  is \emph{periodic} if there exists a positive integer $p$ such that
$\omega_{i+p} = \omega_i$ for all indices $i$, and it is \emph{ultimately periodic} if $\omega_{i+p} = \omega_i$ for all sufficiently large $i$.
An infinite word is \emph{aperiodic} if it is not ultimately periodic.
By a celebrated result due to Hedlund
and Morse \cite{MorHed1940},  a word is ultimately periodic if and only if its factor complexity is uniformly bounded. In particular, $p_{\omega} (n)< n$ for all $n$ sufficiently large.   Words whose factor complexity $\p_{\omega}(n)=n+1$ for all $n\geq 0$ are called
{\it Sturmian words}. Thus, Sturmian words are those aperiodic words having the lowest complexity. Since $\p_{\omega}(1)=2,$ it follows that Sturmian words are binary words. The most extensively studied Sturmian word is the
so-called Fibonacci word
\[{\bf f}=01001010010010100101001001010010
010100101001001010010\cdots\]
fixed by the morphism $0\mapsto 01$ and $1\mapsto 0.$
Let $\omega\in \{0,1\}^\nats$ be a Sturmian word, and let $\Omega$ denote the shift orbit closure of $\omega.$
The condition $\p_{\omega }(n)=n+1$ implies the existence of exactly
one right special and one left special factor of each length. Clearly, given any two left special factors, one is necessarily a prefix of the other. It follows that $\Omega$ contains a unique word all of whose prefixes are left special factors of $\omega.$
Such a word is called the {\it characteristic word} and denoted  $\tilde \omega.$ It follows that both $0\tilde\omega,1\tilde\omega \in \Omega.$  It is readily verified that the Fibonacci word above is a characteristic Sturmian word.  A Sturmian word $\omega $ is called  {\it singular} if $T^n(\omega)= \tilde \omega$ for some $n\geq 1.$ Otherwise it is said to be  {\it nonsingular.}

Sturmian words admit various types of characterizations of
geometric and combinatorial nature. We give two such
characterizations which will be used in the paper: as irrational
rotations on the unit circle and as mechanical words. In \cite
{MorHed1940} Hedlund and Morse showed that each Sturmian word may
be realized measure-theoretically by an irrational rotation on the
circle. That is, every Sturmian word is obtained by coding the
symbolic orbit of a point $x$ on the circle (of circumference one)
under a rotation $R_{\alpha}$ by an irrational angle $\alpha$,
$0<\alpha<1$, where the circle is partitioned into two
complementary intervals, one of length $\alpha $ and the other of
length $1-\alpha .$ And conversely each such coding gives rise to
a Sturmian word. The quantity $\alpha$ is called the {\it slope}.
Namely, the \emph{rotation} by angle $\alpha$ is the mapping
$R_{\alpha}$ from $[0,1)$ (identified with the unit circle) to
itself defined by $R_{\alpha}(x)=\{x+\alpha\}$, where
$\{x\}=x-[x]$ is the fractional part of $x$. Considering a
partition of $[0,1)$ into $I_0=[0, 1-\alpha)$, $I_1=[1- \alpha,
1)$, define a word
$$s_{\alpha, \rho}(n)=\begin{cases}0, & \mbox{ if }
R^n_{\alpha}(\rho)=\{\rho+n\alpha\} \in I_0, \\ 1, & \mbox{ if }
R^n_{\alpha}(\rho)=\{\rho+n\alpha\} \in I_1
\end{cases}$$
One can also define $I'_0=(0, 1-\alpha]$, $I'_1=(1- \alpha, 1]$,
the corresponding word is denoted by $s'_{\alpha, \rho}$. For a
Sturmian word $w$ of slope $\alpha$ its subshift $\Omega$ is given
by $\Omega=\{s_{\alpha, \rho}, s'_{\alpha, \rho}| \rho\in[0,1)\}$.

A straightforward computation shows that
\[s_{\alpha, \rho} (n) = \lfloor \alpha (n+1)+\rho \rfloor -
\lfloor \alpha n+\rho \rfloor,\]  \[s'_{\alpha, \rho} (n) = \lceil
\alpha (n+1)+\rho \rceil - \lceil \alpha n+\rho \rceil;\]
$s_{\alpha, \rho}$ and $s'_{\alpha, \rho}$ are called the {\it
upper} and {\it lower mechanical words} (of slope $\alpha)$ based
at $\rho$.

In \cite{ArRa}  Arnoux and  Rauzy introduced a class
of uniformly recurrent (minimal) sequences $\omega $ on a $m$-letter alphabet of complexity
$\p_{\omega }(n)=(m-1)n+1$  characterized by the following combinatorial
criterion known as the {$\star$ condition:  $\omega $
admits exactly one right special and one left special factor of each
length. We call them {\it Arnoux-Rauzy sequences}.  This condition distinguishes them from other
sequences of complexity
$(m-1)n+1$ such as those obtained by coding trajectories of $m$-interval
exchange transformations. These words are generally regarded as natural combinatorial generalizations of Sturmian words to higher alphabets. In particular, the Fibonacci word generalizes to the $m$-bonacci word fixed by the substitution \[\sigma_m :  \{0,1,\ldots , m-1\} \rightarrow \{0,1,\ldots , m-1\}^*\] given by

 \[ \sigma_m(i) =
\left\{\begin{array}{ll} 0(i+1) \,\,\,&\mbox{for}\,\, 0\leq i <m-1\\
0 \,\,\,&\mbox{for}\,\, i=m-1
\end{array}
\right.
\]

However, many of the dynamical and geometrical interpretations of Sturmian words do not extend to this new class of words (see \cite{CFZ} for example).

In the subsequent sections we will consider partitions of $\nats$ defined by words. Let
$\omega \in {\mathcal A}^\nats ,$  and let ${\mathcal F}$ denote the set of factors of $\omega.$
A finite subset $X$ is called a $\F$-{\it prefix code} if $X\subset \F$ and given any two distinct elements of $ X,$ neither one is a prefix of the other.  A $\F$-prefix code is $\F$-{\it maximal} if it is not properly contained in any other $\F$-prefix code.
The simplest example of a $\F$-maximal prefix code is the set of all elements of $\F$ of some fixed length $d.$
Each $\F$-maximal prefix code $X$ defines a partition

\[\nats = \bigcup _{u\in X} \omega \big|_u\]
If $\omega$ is a Sturmian word, then the corresponding partition is called a {\it Sturmian partition}.

\section{Ultrafilters, IP-sets and central sets}

\subsection{Stone-\v Cech compactification}
Many of our results rely on the algebraic/topological properties of the Stone-\v Cech compactification of $\nats,$ denoted $\beta \nats.$
We regard $\beta \nats$  as the set of all ultrafilters on $\nats$ with the {\it Stone topology.}

\vspace{.1in}

Recall that a set $\mathcal U$ of subsets of $\nats$ is called an {\it ultrafilter} if the following conditions hold:
\begin{itemize}
\item $\emptyset \notin \mathcal U.$
\item If $A\in \mathcal U$ and $A\subseteq B,$ then $B\in \mathcal U.$
\item $A\cap B \in \mathcal U$ whenever both $A$ and $B$ belong to $\mathcal U.$
\item For every $A\subseteq \nats$ either $A\in \mathcal U$ or $A^c\in \mathcal U$ where $A^c$ denotes the complement of $A.$
\end{itemize}

For every natural number $n\in \nats,$ the set $\mathcal{U}_n=\{A\subseteq \nats \,|\, n\in A\}$ is an example of an ultrafilter. This defines an injection $i:\nats \hookrightarrow \beta \nats$ by: $n\mapsto \mathcal{U}_n.$
An ultrafilter of this form is said to be {\it principal.} By way of Zorn's lemma, one can show the existence of  non-principal (or {\it free}) ultrafilters.

It is customary to denote elements of $\beta \nats$ by letters $p,q,r \ldots.$
For each set $A \subseteq \nats,$ we set $A^\circ =\{p\in \beta \nats | A\in p\}.$ Then the set $\mathcal{B}=\{A^\circ | A\subseteq \nats\}$ forms a basis for the open sets (as well as a basis for the closed sets) of $\beta \nats$  and defines a topology on $\beta \nats$ with respect to which $\beta \nats$ is  both compact and Hausdorff.\footnote{Although the existence of free ultrafilters requires Zorn's lemma, the cardinality of $\beta \nats$ is $2^{2^\nats}$ from which it follows that $\beta \nats$ is not metrizable.} 

There is a natural extension of the operation of addition $+$ on $\nats$ to $\beta \nats$ making $\beta \nats$ a compact {\it left-topological semigroup.} More precisely we define addition of two ultrafilters $p,q$ by the following rule:

\[p+q=\{A \subseteq \nats \,|\, \{n\in \nats | A-n \in p\}\in q\}.\]

It is readily verified that $p+q$ is once again an ultrafilter and that
for each fixed $p\in \beta \nats,$ the mapping $q\mapsto p+q$ defines a continuous map from $\beta \nats $ into itself.\footnote{Our definition of addition of ultrafilters is the same as that given in \cite{VB2} but is the reverse of that given in \cite{HS} in which $A\in p+q$ if and only if $\{n\in \nats | A-n \in q\}\in p\}.$ In this case, $\beta \nats$ becomes a compact right-topological semigroup.} The operation of addition in $\beta \nats$ is associative and for principal ultrafilters we have $\mathcal{U}_m + \mathcal{U}_n= \mathcal{U}_{m+n}.$ However in general addition of ultrafilters is highly non-commutative. In fact it can be shown that the center is precisely the set of all principal ultrafilters \cite{HS}.

\subsection{IP-sets and central sets}

Let $(\mathcal{S}, +)$ be a semigroup. An element $p\in \mathcal{S}$  is called an {\it idempotent} if $p+p=p.$
We recall the following result of Ellis \cite{E}:

\begin{theorem}[Ellis \cite{E}]\label{Ellis} Let $(\mathcal{S}, +)$ be a compact left-topological semigroup (i.e., $\forall x\in \mathcal{S}$ the mapping $y\mapsto x+y$ is continuous). Then $\mathcal{S}$ contains an idempotent.
\end{theorem}

\noindent It follows that $\beta \nats$ contains a non-principal ultrafilter $p$ satisfying $p+p=p.$ In fact, we could simply apply Ellis's result to the semigroup $\beta \nats - \mathcal{U}_0.$ This would then exclude the only principal idempotent ultrafilter, namely $\mathcal{U}_0.$ From here on, by an idempotent ultrafilter in $\beta \nats$ we mean a free idempotent ultrafilter.

We will make use of the following striking result due to Hindman linking IP-sets and idempotents in $\beta \nats:$

\begin{theorem}[Theorem 5.12 in \cite{HS}]\label{Hind} A subset $A\subseteq \nats$ is an IP-set if and only if $A\in p$ for some idempotent $p\in \beta \nats.$
\end{theorem}

\noindent It follows immediately that $A$ is an IP$^*$-set if and only if $A\in p$ for every idempotent $p\in \beta \nats$ (see Theorem~2.15 in \cite{VB2}). We also note that the property of being an IP-set is partition regular.\\

In \cite{F}, Furstenberg introduced a special class of IP-sets, called central sets, having additional rich combinatorial properties. They were originally defined in terms of topological dynamics (see Definition~\ref{Cen2}).
As in the case of IP-sets, they may be alternatively defined in terms of belonging to a special class of free ultrafilters, called minimal idempotents\footnote{The equivalence between the two definitions is due to Bergelson and Hindman \cite{BH}.}. To define a minimal idempotent we must first review some basic properties concerning ideals in $\beta \nats$.

Let $(\mathcal{S},+)$ be any semigroup.
Recall that a subset $\mathcal{I}\subseteq \mathcal{S}$ is called a {\it right (resp. left) ideal} if
$\mathcal{I}+\mathcal{S}\subseteq \mathcal{I}$
 (resp. $\mathcal{S}+\mathcal{I}\subseteq \mathcal{I}$). It is called a {\it two sided ideal} if it is both a left and right ideal.
A right (resp. left) ideal $\mathcal{I}$ is called {\it minimal} if every right (resp. left) ideal $\mathcal{J}$ included in $\mathcal{I}$ coincides with $\mathcal{I}.$

Minimal right/left ideals do not necessarily exist  e.g. the commutative semigroup  $(\nats, +)$ has no minimal right/left ideals (the ideals in $\nats$ are all of the form $\mathcal{I}_n=[n, +\infty)=\{m\in \nats\,|\,m\geq n\}.)$ However,
every compact Hausdorff left-topological semigroup $\mathcal{S}$ (e.g., $\beta \nats)$  admits a smallest two sided ideal $K(\mathcal{S})$ which is at the same time the union of all minimal right ideals of $\mathcal{S}$ and the union of all minimal left ideals of $\mathcal{S}$ (see for instance \cite{HS}).
It is readily verified that the intersection of any minimal left ideal with any minimal right ideal is a group. In
particular, there are idempotents in $K(\mathcal{S}).$ Such idempotents are called minimal and their elements are called central sets:

\begin{definition} An idempotent $p$  is called a {\it minimal} idempotent of $\mathcal{S}$ if it belongs to $K(\mathcal{S}).$
\end{definition}

\begin{definition}\label{Cen1} A subset $A \subset \nats$ is called {\it central} if it is a member of some minimal idempotent in $\beta \nats.$ It is called a central$^*$-set if it belongs to every minimal idempotent in $\beta \nats.$
\end{definition}

The equivalence between definitions \ref{Cen2} and \ref{Cen1} is
due to Bergelson and Hindman in \cite{BH}. It follows from the
above definition that every central set is an IP-set and that the
property of being central is partition regular. Central sets are
known to have substantial combinatorial structure. For example,
any central set contains arbitrarily long arithmetic progressions,
and solutions to all partition regular systems of homogeneous
linear equations (see for example \cite{BHS}). Many of the rich
properties of central sets are a consequence of the {\it Central
Sets Theorem} first proved by Furstenberg in Proposition~8.21 in
\cite{F} (see also \cite{DHS, BHS, HS2}). Furstenberg pointed out
that as an immediate consequence of the Central Sets Theorem one
has that whenever $\nats$  is divided into finitely many classes,
and a sequence $(x_n)_{n\in \nats}$ is given, one of the classes
must contain arbitrarily long arithmetic progressions whose
increment belongs to $\{\sum _{n\in F}x_n | F\in \Fin\}.$

\subsection{Limits of ultrafilters}
It is often convenient to think of an ultrafilter $p$ as a $\{0,1\}$-valued, finitely additive probability measure on the power set of $\nats.$ More precisely, for any subset $A\subseteq \nats,$ we say $A$ has $p$-measure $1,$ or is $p$-large if $A\in p.$ This notion of measure gives rise to a notion of convergence of sequences indexed by $\nats$
which is the key tool in allowing us to apply ideas from combinatorics on words to the framework of ultrafilters.  However, from our point of view, it is  more natural to define it alternatively as a mapping from words to words (see Remark~\ref{p-lim}). Let $\mathcal{A}$ denote a non-empty finite set. Then
each ultrafilter $p\in \beta \nats$ naturally defines a mapping \[p^*:\mathcal{A}^\nats \rightarrow \mathcal{A}^\nats\]
as follows:

\begin{definition}\label{p*} For each $p\in \beta \nats$ and $\omega\in \mathcal{A}^\nats,$ we define $p^*(\omega) \in\mathcal{A}^\nats$ by the condition:   $u\in \mathcal{A}^*$ is a prefix of $p^*(\omega) $ $\Longleftrightarrow$ $\omega \big|_u\in p.$
\end{definition}

\noindent We note that if  $u,v\in \mathcal{A}^*, $ $\omega \big|_u, \omega\big|_v \in p$  and  $|v|\geq |u|,$  then $u$ is a prefix of $v.$ In fact, if $v'$ denotes the
prefix of $v$ of length $|u|$ then as $\omega \big|_v \subseteq \omega \big|_{v'}, $ it follows that $\omega \big|_{v'}\in p$ and hence $u=v'.$
Thus $p^*(\omega)$ is well defined.

We note that if $\omega, \nu \in \mathcal{A}^\nats$ and if each prefix $u$ of $\nu$ is a factor of $\omega, $ then there exists an ultrafilter $p\in \beta \nats$ such that $p^*(\omega)=\nu.$ In fact, the set
\[\mathcal{C}=\{\omega \big|_u \,|\, u \,\,\mbox{is a prefix of}\,\,\nu\}\]
satisfies the finite intersection property, and hence by a routine argument involving Zorn's lemma it follows that there exists  a $p\in \beta \nats$ with $\mathcal{C}\subseteq p.$

It follows immediately from the definition of $p^*,$ Definition~\ref{Cen1}  and Theorem~\ref{Hind} that

\begin{lemma}\label{IP} The set $\omega\big|_u$ is an IP-set (resp. central set) if and only if $u$ is a prefix of $p^*(\omega)$ for some idempotent (resp. minimal idempotent) $p\in \beta \nats.$
\end{lemma}

\begin{lemma} For each $p\in \beta \nats,$ $\omega \in \mathcal{A}^\nats$ and $u\in \mathcal{A}^*$ we have
\[ p^*(\omega)\big|_u=\{m\in \nats\,|\, \omega\big|_u-m\in p\}\]
where $\omega\big|_u-m$ is defined as the set of all $n\in \nats$ such that $n+m\in \omega\big|_u.$

\end{lemma}

\begin{proof} Suppose $m\in p^*(\omega) \big|_u.$ Then by definition $u$ occurs in position $m$ in $p^*(\omega).$
Let $v$ denote the prefix of $p^*(\omega)$ of length $|v|=m+|u|.$ Then, as $u$ is a suffix of $v$ we have  $\omega \big|_v +m \subseteq \omega \big|_u$  and hence $\omega \big|_v \subseteq \omega \big|_u -m.$ But as
$v$ is a prefix of $p^*(\omega)$ we have $\omega \big|_v \in p$ and hence $\omega \big|_u -m \in p$ as required.

Conversely, fix $m\in \nats$ such that $\omega \big|_u-m\in p.$ Let $Z$ be the set of all factors $v$ of $\omega$ of length $|v|=m+|u|$ ending in $u.$  Then \[\omega \big|_u-m \subseteq \bigcup_{v\in Z} \omega\big|_v.\]
It follows that there exists $v\in Z$ such that $\omega \big|_v\in p.$ In other words, there exists $v\in Z$ such that
$v$ is a prefix of $p^*(\omega).$ It follows that $u$ occurs in position $m$ in $p^*(\omega).$
\end{proof}

\begin{lemma}\label{idempotent} For $p,q \in \beta \nats$ and $\omega\in \mathcal{A}^\nats,$ we have $(p+q)^*(\omega)=q^*(p^*(\omega)).$ In particular, if $p$ is an idempotent, then $p^*(p^*(\omega))=p^*(\omega).$

\end{lemma}

\begin{proof} For each word $u\in \mathcal{A}^*$ we have that $u$ is a prefix of $(p+q)^*(\omega)$ if and only if
\[\omega \big|_u \in p+q \Longleftrightarrow \{ m \in \nats \,|\, \omega \big|_u -m \in p\} \in q.\]
On the other hand, $u$ is a prefix of $q^*(p^*(\omega))$ if and only if $p^*(\omega)\big|_u \in q.$
The result now follows immediately from the preceding lemma.
\end{proof}

\begin{lemma}\label{commutes}For each $p\in \beta \nats$ and $\omega \in \mathcal{A}^\nats$ we have $p^*(T(\omega))=T(p^*(\omega))$ where $T:\mathcal{A}^\nats \rightarrow \mathcal{A}^\nats$ denotes the shift map.
\end{lemma}

\begin{proof} Assume $u\in \mathcal{A}^*$ is a prefix of $p^*(T(\omega)).$ Then $T(\omega)\big|_u \in p.$
But
\[T(\omega)\big|_u =\bigcup_{a\in \mathcal{A}}\omega\big|_{au}.\]
It follows that there exists $a\in \mathcal{A}$ such that $\omega\big|_{au}\in p.$ Thus $au$ is a prefix of $p^*(\omega)$ and hence $u$ is a prefix of $T(p^*(\omega)).$
\end{proof}

\begin{remark}\label{p-lim}
{\rm It is readily verified that our definition of $p^*$ coincides with that of $p$-$\lim_n.$
More precisely, given a sequence $(x_n)_{n\in \nats}$ in a topological space and an ultrafilter $p\in \beta \nats,$  we write
$p$-$\lim_nx_n=y$ if for every neighborhood $U_y$ of $y$ one has $\{n\,| x_n \in U_y\}\in p.$
In our case we have
$p^*(\omega)=p$-$\lim_n(T^n(\omega))$  (see \cite{H2}). With this in mind, the preceding two lemmas are well known (see for instance \cite{Bl,H2}).
However, our defining condition of $p^*$ in Definition~\ref{p*} does not directly rely on the topology and so may be applied in other general settings. For instance, let $\Omega \subseteq \mathcal{A}^\nats$ be a subshift, and  $\mathcal{N}=\{n_0<n_1<n_2<\cdots \}$ an infinite sequence of natural numbers. For each  $\omega \in \Omega$ we put
\[X_k^{\mathcal{N}}=\{\omega_{n+n_0}\omega_{n+n_1}\ldots \omega_{n+n_{k-1}}\,|\, n\geq 0\}\subseteq \mathcal{A}^k.\]
For each $u\in X_k^{\mathcal{N}}$ we define the set \[\omega ^{\mathcal{N}}\big|_u=\{n\in \nats\,|\, \omega_{n+n_0}\omega_{n+n_1}\ldots \omega_{n+n_{k-1}}=u\}.\] Then the sets $\omega ^{\mathcal{N}}\big|_u$ with $u\in X_k^{\mathcal{N}}$ partition $\nats.$ So, given $p\in \beta \nats,$ for each $k\geq 1$ there exists a unique $u\in X_k^{\mathcal{N}}$ with $\omega ^{\mathcal{N}}\big|_u\in p.$ Moreover if $v\in X_{k+1}^{\mathcal{N}}$ and $\omega ^{\mathcal{N}}\big|_v \in p,$ then $u$ is a prefix of $v.$ So using the condition in Definition~\ref{p*}, each
infinite sequence $\mathcal{N}$ and ultrafilter $p\in \beta \nats$ defines a mapping $\Omega \rightarrow \Omega.$ Of particular interest is the case in which $\Omega$ is a uniform set in the sense of T. Kamae  and $\mathcal{N}$ is chosen such that $\omega[\mathcal{N}]$ is a super-stationary set (see \cite{K1,K2}).

Another situation in which the defining condition of Definition~\ref{p*} applies is in the context of infinite permutations \cite{FA}. By an infinite permutation $\pi$ we mean a linear ordering on $\nats.$ Then for each finite permutation $u$ of $\{1,2,\ldots ,n\}$   we say that $u$ occurs in position $m$ of $\pi$ if the restriction of $\pi$ to $\{m, m+1, \ldots , m+n-1\}$ is equal to $u.$ Thus we may define the set $\pi\big|_u$  as the set of all $m\in \nats$ such that $u$ occurs in position $m$ in $\pi,$  and again the sets $\pi\big|_u$ (over all permutations $u$ of  $\{1,2,\ldots ,n\}$) determine a partition of $\nats.$ Hence each $p\in \beta \nats$ defines a map from the set of all infinite permutations into itself.}
\end{remark}

 In what follows, we will make use of the following key result in \cite{HS} (see also Theorem 1 in \cite{Bl} and Theorem 3.4 in \cite{VB2}):

 \begin{theorem}[Theorem 19.26 in \cite{HS}]\label{Berg} Let $(X,T)$ be a topological dynamical system. Then if two points $x,y \in X$ are proximal  with $y$ uniformly recurrent, then  there exists a minimal idempotent $p\in \beta \nats$ such that $p^*(x)=y.$
 \end{theorem}

 \noindent As a consequence we have

 \begin{theorem}\label{IPC} Let $\omega \in \mathcal{A}^\nats$ be a uniformly recurrent word, and let $u \in \mathcal{A}^+.$ Then $\omega \big|_u$ is an IP-set if and only if $\omega \big|_u$ is a central set.
 \end{theorem}

\begin{proof} For any $A\subset \nats$ we have that if $A$ is central then $A$ belongs to some minimal idempotent $p\in \beta \nats$ and hence in particular $A$ belongs to an idempotent in $\beta \nats.$ Hence by
Theorem~\ref{Hind} we have that $A$ is an IP-set. Now suppose that $\omega \big| _u$ is an IP-set. Then $\omega \big|_u$ belongs to some idempotent $p\in \beta \nats.$ Set $\nu=p^*(\omega).$  Then $u$ is a prefix of $\nu.$
Also, since $p$ is idempotent we have $p^*(\nu)=p^*(p^*(\omega))=p^*(\omega)=\nu.$ Hence for every prefix $v$ of $\nu$ we have that $\nu \big|_v\in p$ and $\omega \big |_v\in p$ and hence $\nu \big|_v \cap \omega \big |_v\in p.$ In particular $\nu \big|_v \cap \omega \big |_v\neq \emptyset.$ Hence $\omega$ and $\nu$ are proximal. Since $\omega$ is uniformly recurrent, it follows that $\nu$ is also uniformly recurrent. Hence by Theorem~\ref{Berg} there exists a minimal idempotent $q$ with $q^*(\omega)=\nu.$ Hence $\omega \big |_u \in q,$ whence $\omega \big|_u$  is central.
\end{proof}

\begin{remark}\rm{A special case of  Theorem~\ref{Berg} states that if $x$ and $y$ are uniformly recurrent infinite words, then  $x$ and $y$ are  proximal if and only if $p^*(x)=y$ for some idempotent ultrafilter $p \in \beta \nats.$
In the case of binary words, we could consider the following alternative notion: We say that $x$ and $y$ are {\it anti-proximal} if the set $\{n\in \nats\,|\, x_n\neq y_n\}$ is thick. For example the two fixed points $\mathbf{t}_0$ and $\mathbf{t}_1$ of the Thue-Morse morphism are anti-proximal. In \cite{BHPZ}, together with N. Hindman  we show that for every prefix $u$ of $\mathbf{t}_1,$ the set $\mathbf{t}_0\big|_u$ is finite FS-big. We recall that $A\subseteq \nats$ is {\it finite FS-big} if $\forall k$ there exists $(x_i)_{i=1}^k$ such that $\mbox{FS}(x_i)_{i=1}^k \subseteq A$ where
\[\mbox{FS}(x_i)_{i=1}^k=\{\sum_{i\in F}x_i\,|\, F\subseteq\{1,2,\ldots,k\}\}.\]
As in the case of  IP-sets, the property of being finite FS-big is partition regular, i.e., if $A\subseteq \nats$ is
finite FS-big and $A=\bigcup _{i=1}^rA_i,$ then some $A_i$ is finite FS-big (see \cite{BHPZ}).
In the context of binary words, the notions of proximality and anti-proximality are somewhat similar in the sense that in both cases the behavior of one word is strongly affected by the behavior of the other: In case $x$ and $y$ are proximal, then $x$ does as $y$ on a thick set while if $x$ and $y$ are anti-proximal, then $x$ and $y$ play opposites on a thick set. One might ask the question of finding an analogue  of Theorem~\ref{Berg} characterizing anti-proximality.
}\end{remark}

\section{A first analysis of some concrete examples}

\subsection{The Fibonacci word}
While most of the proofs of the results announced in the Introduction rely on the algebraic and topological properties of ultrafilters on $\nats$  and their links to IP-sets,  we begin by analyzing concretely a few examples generated by simple substitution rules. To establish that certain subsets of $\nats$ are IP-sets, we will use nothing more than the definition of IP-sets and the abstract numeration systems defined by substitutions first introduced by J.-M. Dumont and A. Thomas \cite{DT1, DT2}.

Let us begin with the  {\it Fibonacci} infinite word $\mathbf{f}=f_0f_1f_2\ldots \in \{0,1\}^\nats$ given by
\[{\bf f}=01001010010010100101001001010010
010100101001001010010\cdots\]

We set \[\mathbf{f}\big|_0=\{n\in \nats | f_n=0\}\] and
\[\mathbf{f}\big|_1=\{n\in \nats | f_n=1\}.\] So
$\mathbf{f}\big|_0=\{0,2,3,5,7,8,10,11,13,15,16, \ldots\}$ and
$\mathbf{f}\big|_1=\{1,4,6,9,12,14,17,\ldots\}.$ This defines the
Sturmian partition $\nats =\mathbf{f}\big|_0 \cup
\mathbf{f}\big|_1.$ Let us denote by $F_n$ the $n$th Fibonacci
number so that $F_0=1, F_1=2,F_2=3,\ldots .$ It is well known that
each positive integer $n$ has one or more representations when
expressed as a sum of distinct Fibonacci numbers, i.e.,
$n=\sum _{i=0}^k t_iF_i$  with $t_i\in \{0,1\}$ and $t_k=1.$
We call the associated $\{0,1\}$-word $t_kt_{k-1}\cdots t_0$ a representation of $n.$
For example, for $n=50$ we obtain the following $6$ representations
(arranged in
decreasing lexicographic order):
\[
\begin{array}{r}
10100100\\
10100011\\
10011100\\
10011011\\
1111100\\
1111011
\end{array}
\]
The lexicographically largest representation is  obtained  by applying the {\it greedy
algorithm}. This gives rise to a representation of $n$ of the form
$n=\sum _{i=0}^k t_iF_i$  with
$t_{i+1}t_{i}\neq 11$ for each $0\leq i \leq k-1.$ This
representation of $n$ is  called the {\it
Zeckendorff representation} \cite{Zeck} (a special case of the Dumont-Thomas numeration system \cite{DT1,DT2}). We shall write
 $\Zcal(n)=t_kt_{k-1}\ldots t_0.$
It follows immediately that $\Zcal(F_n)=10^n.$
The connection between $\Zcal(n)$  and the entry $f_n$ of the Fibonacci word $\mathbf{f}$ is given by the following well known fact:
$f_n=0$ whenever $\Zcal(n)$ ends in $0$ and $f_n=1$ whenever $\Zcal(n)$ ends in $1.$ Thus
 \[\mathbf{f}\big|_0=\{n\in \nats \,|\, \Zcal(n)\,\mbox {ends in}\, 0\}\] and \[\mathbf{f}\big|_1=\{n\in \nats \,| \, \Zcal(n)\,\mbox {ends in}\, 1\}.\]

We now consider the sequence
$(x_n)_{n\in \nats}$ given by $x_n= F_{2n+1.}$
It is readily verified that for each $A\in \Fin,$ the Zeckendorff
representation of $\sum _{n\in A}x_n$ ends in $10^{2m+1}$ where $m=\mbox{min}(A).$ In fact,  the symbolic sum of the individual Zeckendorff representations of each $x_n$ occurring in  $\sum _{n\in A}x_n$  does not involve any carry overs. Moreover the resulting expression does not contain any occurrences of $11$ and hence is equal to the Zeckendorff representation of $\sum _{n\in A}x_n.$  Thus every finite sum of the form $\sum _{n\in A}x_n$ with $A\in \Fin$ belongs to $\mathbf{f}\big|_0.$  Thus we have shown that $\mathbf{f}\big|_0$ is an IP-set.

We next verify that $\mathbf{f}\big|_1$ is not an IP-set, and hence  $\mathbf{f}\big|_0$ is an  IP$^*$-set.
We will use the following general observation. Consider a subset
$A\subset\nats$
partitioned into $k>0$ non-intersecting sets: $A=A_1\cup A_2 \cup
\cdots \cup A_k$. Suppose that for each $1\leq j \leq k$  there exists a positive integer $N$ (which may
depend on $j$) such that whenever $m_1, m_2,\ldots , m_N$ are
distinct elements of $ A_j,$
we have $\sum _{i=1}^{N} m_i \notin A$. Then $A$ is not an IP-set.
In fact, if $A$ were an IP-set, then for some $1\leq j\leq k,$  there would exist a sequence
$x_1<x_2<x_3<\cdots$ contained in $A_j$ such that  $\{\sum _{n\in F}x_n | F\in \Fin\}\subset A.$ \\

Let   $\alpha = \frac{3-\sqrt{5}}{2}.$ Then the Fibonacci word
$\mathbf{f}$ is the orbit of the point $\alpha$ under irrational
rotation $R_{\alpha}$ on the unit circle by $\alpha.$ Let $I$ be
the interval $[1-\alpha, 1)$ (the interval coded by $1$). So $n\in
\mathbf{f}\big|_1$ if and only if $R^n_{\alpha}(\alpha)=\{\alpha +
n\alpha\} =\{(n+1)\alpha\} \in I$.

\noindent Fix \[(1-\alpha)/3\leq \alpha' \leq (1-\alpha)/2\] and put
\[I_1=[1-\alpha, 1-\alpha')\,\,\,\,\mbox{and}\,\,\,\,I_2=[1-\alpha', 1).\] Since
$\alpha '\leq (1-\alpha)/2$ it follows that $\alpha'<\alpha.$ Also
for $j=1,2$ set \[A_j=\{n\in \nats \,|\, R^n(\alpha) \in I_j\}.\]  Thus
$A_1,A_2$ partitions the set $\mathbf{f}\big|_1.$ We now show that
$\mathbf{f}\big|_1$ is not an IP-set by showing that  the sum of
any three elements of $A_1$ belongs to $\mathbf{f}\big|_0$ and
that  the sum of any two elements of $A_2$ belongs to
$\mathbf{f}\big|_0.$

Now take any $n_1, n_2, n_3 \in A_1$ and set \[x_1= \{
(n_1+1)\alpha \}, \,\,\,x_2= \{ (n_2+1)\alpha \},\,\,\, x_3= \{ (n_3+1)\alpha
\} .\] Then $x_1,x_2,x_3 \in [1-\alpha, 1-\alpha')$ and $n_1+n_2+n_3$ corresponds to
the point \[\{ (n_1+n_2+n_3+1)\alpha \}=\{ x_1+x_2+x_3-2\alpha \}.\]
Since $x_1, x_2, x_3 \in [1-\alpha, 1-\alpha')$, we have
\[\{x_1+x_2+x_3-2\alpha\}\in [\{3-5\alpha
\},\{3-3\alpha'-2\alpha\}).\] Since $\alpha'\geq
\frac{1-\alpha}{3}$ it follows that
 \[2-3\alpha'-2\alpha \leq 1-\alpha,\] and hence \[\{2-3\alpha'-2\alpha \}\leq 1-\alpha,\] which gives
 \[\{3-3\alpha'-2\alpha \}\leq 1-\alpha\] as required.

Similarly take any $n_1, n_2 \in A_2.$ Set \[x_1= \{ (n_1+1)\alpha
\},\,\,\, x_2= \{ (n_2+1)\alpha \}\]
so that $x_1,x_2 \in [1-\alpha', 1)$. Then $n_1+n_2$
corresponds to the point \[\{ (n_1+n_2+1)\alpha \}=\{
x_1+x_2-\alpha \}.\] Since $x_1, x_2 \in [1-\alpha',1)$, we have
\[\{x_1+x_2-\alpha\}\in [\{2-2\alpha'-\alpha\},1-\alpha).\] Since
\[\alpha'\leq \frac{1-\alpha}{2}\] it follows that
\[\{1-2\alpha'-\alpha \}\geq 0,\] and hence
\[\{2-2\alpha'-\alpha\}\geq 0.\]

The above arguments may be generalized to show that $\mathbf{f}\big|_u$ is an IP$^*$-set for every prefix $u$ of $\mathbf{f}.$\\

In contrast, let us consider the sets $\mathbf{g}\big|_0$ and $\mathbf{g}\big|_1$ where $\mathbf{g}=0\mathbf{f} = 001001010010010\ldots .$
Thus, \[\mathbf{g}\big|_0=\{n\in \nats \,|\, g_n=0\}=\{0\} \cup \{n\geq 1 \,|\, f_{n-1}=0\}.\]
Consider the sequence $(y_n)_{n\in \nats}$ defined by $y_n=F_{2n+2}.$ It is readily verified that
$\Zcal(y_n-1)=(10)^{n+1}$ and hence each $y_n$ belongs to $\mathbf{g}\big|_0.$
Now fix $A\in \Fin.$ Since the Zeckendorff
representation of $\sum _{n\in A}y_n$ ends in $10^{2m+2}$ where $m=\mbox{min}(A),$ it follows that
$\Zcal(\sum _{n\in A}y_n -1)$ ends in $(10)^{m+1},$ and hence $\sum _{n\in A}y_n \in \mathbf{g}\big|_0.$
Thus, $\mathbf{g}\big|_0$ is an IP-set.
Similarly, it is readily verified that for each $A\in \Fin,$ we have that $\sum _{n\in A}x_n\in \mathbf{g}\big|_1$ where $x_n=F_{2n+1}.$  Thus this time we obtain the Sturmian decomposition $\nats =\mathbf{g}\big|_0 \cup \mathbf{g}\big|_1$ in which both sets $\mathbf{g}\big|_0$ and  $\mathbf{g}\big|_1$ are IP-sets, and hence central sets. In this case, neither
$\mathbf{g}\big|_0$ nor $\mathbf{g}\big|_1$ is an IP$^*$-set.  Once again, these arguments may be extended to show that both $\mathbf{g}\big|_{0u}$ and $\mathbf{g}\big|_{1u}$ are central sets for any prefix $u$ of $\mathbf{f}$ and hence neither set is an IP$^*$-set.

\vspace{.1 in}

\noindent In summary, by Theorem~\ref{IPC} we have:

\begin{proposition} Let $\mathbf{f}$ denote the Fibonacci word. Then for every prefix $u$ of $\mathbf{f}$ the set
$\mathbf{f}\big|_u$ is an IP$^*$-set (and hence a central$^*$ set). Setting $\mathbf{g}=0\mathbf{f}$ we have that for every prefix $u$ of $\mathbf{f}$ the sets $\mathbf{g}\big|_{0u}$ and $\mathbf{g}\big|_{1u}$ are both IP-sets (resp. central sets).
\end{proposition}

\subsection{The $m$-bonacci word}
The above analysis extends more generally to the so-called $m$-bonacci word. Fix a positive integer $m\geq 2,$ and
let $\mathbf{t}=t_0t_1t_2\ldots \in \{0,1,\ldots , m-1\}^\nats$ denote the {\it $m$-bonacci infinite word} fixed by the substitution \[\sigma_m :  \{0,1,\ldots , m-1\} \rightarrow \{0,1,\ldots , m-1\}^*\] given by

 \[ \sigma_m(i) =
\left\{\begin{array}{ll} 0(i+1) \,\,\,&\mbox{for}\,\, 0\leq i <m-1\\
0 \,\,\,&\mbox{for}\,\, i=m-1
\end{array}
\right.
\]

\noindent Using the associated Dumont-Thomas numeration system, we will show:

\begin{proposition}\label{mbonacci} Let $m\geq 2,$ and consider the partition of $\nats$ given by
\[\nats = \bigcup_{0\leq k \leq m-1} \mathbf{g}\big|_k\]
where $\mathbf{g}=0\mathbf{t}\in  \{0,1,\ldots , m-1\}^\nats.$ Then for each $0\leq k \leq m-1$ the set $\mathbf{g}\big|_k$ is an IP-set (resp. central set).
\end{proposition}

The proof is a simple  extension of the ideas outlined above in the case of the Fibonacci word.
For each $m\geq 2,$ we define the $m$-bonacci numbers by
$T_k=2^k$ for $0\leq k\leq m-1$ and $T_k=T_{k-1} + T_{k-2} +\cdots +T_{k-m}$
for $k\geq m.$ When $m=2,$ these are the usual Fibonacci numbers.
Each positive integer $n$ may be written
 in one or more ways in the form
 $n=\sum_{i=1}^k t_iT_{k-i}$ where $t_i\in \{0,1\}$ and $t_1=1.$
By applying the greedy algorithm,
one obtains a representation of $n$ of the form $w=t_1t_2\cdots t_k$ with the
property that
$w$ does not contain $m$ consecutive $1$'s.
Such a representation of $n$ is necessarily unique and is called the $m$-{\it Zeckendorff
representation}
of $n,$ denoted $\Zcal_m(n)$ (see \cite{EZ}).  Thus $\Zcal_m(T_n)=10^n$ for $n\geq 0.$

\begin{proof} Fix $0\leq k \leq m-1.$ We will show that the set $\mathbf{g}\big|_k$ is an IP-set.
 It is well known that
$t_n=k$ if and only if $\Zcal_m(n)$ ends in $01^k.$
Hence
\[\mathbf{g}\big|_k=\{ n\in \nats \,|\, g_n=k\} = \{n\in \nats \,|\, t_{n-1}=k\} = \{n\in \nats \,|\, \Zcal_m(n-1)\,\mbox{ends in}\, 01^k\}.\]

Consider the sequence $(x_n)_{n\in \nats}$ given by $x_n=T_{mn+k}.$ It is readily verified for any finite subset $A\subset \nats,$  the $m$-Zeckendorff
representation of the finite sum $s=\sum_{n\in A}x_n$ ends in
$10^{mr+k}$ where $r=\mbox{min}(A)$ and hence the $m$-Zeckendorff
representation of $s-1$ ends in $(1^{m-1}0)^r1^k$ and hence $s\in \mathbf{g}\big|_k$ as required.

\noindent Having established that each of the sets $\mathbf{g}\big|_k$ is a central set (for $0\leq k\leq m-1),$ it follows that no  $\mathbf{g}\big|_k$ is an IP$^*$-set.

\end{proof}

\section{Sturmian partitions \& central sets}

We now study more generally partitions of $\nats$ generated by Sturmian words and prove theorems~\ref{sturm1} and \ref{sturm2}.
Throughout this section $\omega=\omega_0\omega_1\omega_2\ldots  \in \{0,1\}^\nats$ will denote a Sturmian word,
$\F$  the set of all factors of $\omega,$
and $(\Omega, T)$ the subshift generated by $\omega,$ where $T$ denotes the shift map.  We denote by
 $\tilde \omega \in \Omega$ the characteristic word.
 \begin{lemma}\label{unique} If $\omega, \omega ', \omega'' \in \Omega$ are such that $T^{n_0}(\omega)=T^{n_0}(\omega')=T^{n_0}(\omega''),$ then Card$\{\omega, \omega ', \omega''\}\leq 2.$
\end{lemma}

\begin{proof}This follows immediately from the fact that $\Omega$ contains a unique characteristic word and that
this word is aperiodic.
\end{proof}

  \noindent We will make use of the following key lemma which essentially says that two distinct Sturmian words $\omega $ and $\omega ' $ are proximal if and only if $T^n(\omega)=T^n(\omega')=\tilde \omega$ for some $n\geq 1.$

 \begin{lemma}\label{key} Let $\omega $ and $\omega ' $ be distinct  elements of $ \Omega .$ Then either $T^n(\omega)=T^n(\omega')=\tilde \omega$ for some $n\geq 1,$   or there exists $N>0$ such that $\omega_n\omega_{n+1}\ldots \omega_{n+N}\neq \omega'_n\omega'_{n+1}\ldots \omega'_{n+N}$ for every $n\in \nats .$
 \end{lemma}

 \begin{proof}We will use a definition of Sturmian words via
 rotations, which we recalled in Section 2. Notice that $\tilde \omega=s_{\alpha, \alpha}=s'_{\alpha, \alpha}$, and
singular words correspond to the case when the orbit of a point
under rotation map goes through the point $\alpha$. If $s_{\alpha,
\rho}$ is non-singular, then $s_{\alpha, \rho}=s'_{\alpha, \rho}$.
If $w \neq w'$ are singular words defined by rotations of the same
point, i. e., $w=s_{\alpha, \rho}$, $w'=s'_{\alpha, \rho}$, then
they differ only when they pass through $1-\alpha$ and $0$, i. e.,
in maximum two points, so there exists $n_0\geq1$ such that
$T^{n_0}(\omega)=T^{n_0}(\omega')=\tilde \omega$.

Now consider the case when $w$, $w'$ are defined by rotations of two
different points $\rho$, $\rho'$, $0\leq\rho<\rho'<1$. To be definite, let
us consider the interval exchange of $I_0$ and $I_1$ for both $w$ and
$w'$. We should prove that there there exists $N>0$ such that
\[\omega_n\omega_{n+1}\ldots \omega_{n+N}\neq \omega'_n\omega'_{n+1}\ldots
\omega'_{n+N}\] for every $n\in \nats .$ We have $w_i \neq w'_i$ if and only if
$w_i\in I_0$, $w'_i\in I_1$ or $w_i\in I_1$, $w'_i\in I_0$. This condition
is equivalent to \[w_i\in [1-\alpha-(\rho'-\rho), 1-\alpha)\cup
[1-(\rho'-\rho), 1).\] The distribution of points from the orbit of any
point
$\theta$ under rotation by $\alpha$ is dense, it means that for every
$\epsilon$ there exists $N(\epsilon)$, such that after
$N(\epsilon)$ iterations points split the interval $[0,1)$ into
intervals of length less than $\epsilon$. Putting
$\epsilon=\rho'-\rho$, we get that every $N=N(\epsilon)$ consecutive
iterations there will be a point in every interval of length $\rho'-\rho$,
so there are points in $[1-\alpha-(\rho'-\rho),1-\alpha)$ and
$[1-(\rho'-\rho), 1)$ every $N$ iterations, and hence for every $n$ there
exists $i\in [n, n+N-1]$ with $w_i\neq w'_i$.

 \end{proof}

 \noindent We first consider the case of nonsingular Sturmian words:

 \begin{lemma}\label{fix} Let $\omega\in \{0,1\}^\nats$ be a nonsingular Sturmian word and $p\in \beta \nats$ an idempotent ultrafilter. Then $p^*(\omega)=\omega.$  \end{lemma}

 \begin{proof} Suppose to the contrary that $p^*(\omega)\neq \omega.$  Then since $\omega$ is nonsingular,  Lemma~\ref{key} implies that for all sufficiently long factors $u$ of $\omega,$ we have that $\omega \big|_u \cap p^*(\omega)\big|_u = \emptyset .$  But, by Lemma~\ref{idempotent} we have
 $p^*(p^*(\omega))=p^*(\omega),$ that is the image under $p^*$ of $\omega$ and $p^*(\omega)$ coincides. It follows by definition of $p^*$ that for every prefix $u$ of $p^*(\omega)$ we have $\omega \big|_u \in p$ and $p^*(\omega)\big|_u \in p$ and hence $\omega \big|_u \cap p^*(\omega)\big|_u \in p,$  a contradiction.
 \end{proof}

 \begin{proof}[Proof of Theorem~\ref{sturm1}]Let $\omega $ be a nonsingular Sturmian word,  $u$ a prefix of $\omega,$ and $p\in \beta \nats$ an idempotent ultrafilter.  Then by Lemma~\ref{fix} $u$ is a prefix of $p^*(\omega)$ and hence
 $\omega \big|_u \in p.$ Thus for each prefix $u$ of $\omega$ the set $\omega \big|_u$ belongs to every idempotent ultrafilter and hence is an IP$^*$-set. It follows that if $v\in F$ is not a prefix of $\omega,$ then $\omega\big|_v$ is not an IP-set. Finally, let $v$ be any factor of $\omega$ and $n\in \nats.$ Then $\omega \big|_v-n = T^n(\omega)\big|_v.$  If $n\in \omega \big|_v,$ then $v$ is a prefix of $T^n(\omega)$ from which it follows that
 \[\omega \big|_v-n = T^n(\omega)\big|_v \in p.\]
 Hence $\omega \big|_v-n$ is an IP$^*$-set
 \end{proof}

\noindent As a consequence of the above theorem we have

\begin{corollary}\label{C2} Let $\omega$ and $\omega'$ be two nonsingular Sturmian words, not necessarily of the same slope. Then for every prefix $u$ of $\omega$ and every prefix $u'$ of $\omega'$ we have that $\omega\big|_u \cap \omega'\big|_{u'}$ is an IP$^*$-set (resp. central$^*$ set), in particular the intersection is infinite.
\end{corollary}

\noindent We note that the assumption that $\omega$ and $\omega '$ be nonsingular is necessary, as for example
if we consider $\omega = 0\mathbf{f}$ and $\omega'=1\mathbf{f}$ with $\mathbf{f}$ the Fibonacci word,
then $\omega\big|_0 \cap \omega'\big|_{1}=\{0\}.$
\vspace{.1in}

 \begin{proof} Let $\omega$ and $\omega'$ be two nonsingular Sturmian words,  $u$ a prefix of $\omega,$  $u'$ a prefix of $\omega',$ and $p\in \beta \nats$ an idempotent ultrafilter. Then by Corollary~\ref{C1} we have that $\omega\big|_u \in p$ and $\omega\big|_{u'}\in p$ and hence $ \omega\big|_u \cap \omega\big|_{u'}\in p.$ Thus $ \omega\big|_u \cap \omega\big|_{u'}$ belongs to every idempotent and hence is an IP$^*$-set.
 \end{proof}

\noindent  We next consider singular Sturmian words.

 \begin{lemma}Let  $\omega, \omega'\in \Omega$ be distinct Sturmian words such that $T^{n_0}(\omega) = T^{n_0}(\omega')=\tilde \omega$ for some $n_0\geq 1.$ Then for every $u\in  \F$ and every non-principal ultrafilter $p\in \beta \nats$ we have
 \[\omega\big|_u \in p \Longleftrightarrow \omega'\big|_u \in p.\]
 In particular, $p^*(\omega)=p^*(\omega ').$
 \end{lemma}

 \begin{proof} Since $p$ is a non-principal ultrafilter, we have that
 $\omega \big|_u \in p \Longleftrightarrow \omega \big|_u \cap [N,+\infty) \in p$ for all $ N\geq 1.$ Similarly
 $\omega' \big|_u \in p \Longleftrightarrow \omega' \big|_u \cap [N,+\infty) \in p$ for all $ N\geq 1.$
 But for each $u\in \F,$ we have $\omega \big|_u \cap [n_0,+\infty) =\omega' \big|_u \cap [n_0,+\infty) .$ The result now follows.
 \end{proof}

 \begin{lemma}\label{choice} Let $\omega,\omega' \in \Omega$ be as in the previous lemma, and let $p\in \beta \nats$ be an idempotent ultrafilter. Then $p^*(\omega)=p^*(\omega') \in\{\omega, \omega'\}.$
 \end{lemma}

 \begin{proof} That $p^*(\omega)=p^*(\omega')$ follows from the previous lemma and the fact that idempotent ultrafilters are non-principal (see for instance \cite{VB2}). By Lemma~\ref{commutes},  $p^*$ commutes with the shift map $T,$ and hence
 \[T^{n_0}p^*(\omega)=p^*(T^{n_0}\omega)=p^*(\tilde \omega)=\tilde \omega\] where the last equality follows from Lemma~\ref{fix}. By Lemma~\ref{unique} applied to $\omega''=p^*(\omega)$ it follows that $p^*(\omega)=\omega$ or $p^*(\omega)=\omega'.$
 \end{proof}

 \begin{proof}[Proof of Theorem~\ref{sturm2}]Let $\omega \in \Omega $ and $n_0$ be as in the statement of the theorem.  Then there exists a unique $\omega'\in \Omega$ with $\omega'\neq \omega$ and with $T^{n_0}(\omega')=\tilde \omega.$ Suppose that $\omega\big|_u$ is an IP-set for some  $u\in \F.$ Then
 by Lemma~\ref{IP} it follows that $u$ is a prefix of $p^*(\omega)$ for some idempotent ultrafilter $p\in \beta \nats.$   It follows from Lemma~\ref{choice} that $u$ is a prefix of $\omega$ or a prefix of $\omega'.$ This proves one direction.

 To establish the other direction, we must show that $\omega \big|_u$ is a central set for each prefix $u$ of $\omega$ or of  $\omega'.$ By Theorem~\ref{Berg}, there exist minimal idempotent ultrafilters $p_1,p_2\in \beta \nats$ such that
 $p_1^*(\omega)=\omega$ and $p_2^*(\omega)=\omega'.$  The result now follows.
 \end{proof}

 \begin{remark}\rm{V. Bergelson \cite{BergCom} suggested to us that the above result may be related to a previously known partition of $\mathbb{N}$ into
two  central sets $X = \{ [mx], m \in \mathbb{N}\}$ and $Y = \{[my], m
\in \mathbb{N}\}$, where $x$ and $y$ are two irrational numbers
satisfying $1/ x + 1/y = 1.$  In fact,  this partition precisely
corresponds to our partition of $\mathbb{N}$ into two IP-sets
$\omega\big|_0$ and $\omega\big|_1$ where $\omega$ is of the form
$0\tilde \omega$ and $\tilde \omega$ is a characteristic Sturmian.

This could be seen using the definition of Sturmian words via
mechanical words (see Section 2 for notation).  For a slope
$\alpha$ we have $s_{\alpha, 0} = 0 \tilde \omega$. Let
$\alpha=1/x$ and $1/y= 1-\alpha;$ then $s_{\alpha, 0}(n)=1$ if and
only if there exists an integer $k$ such that $\alpha (n+1) \geq
k$ and $ \alpha n <k$. It is easy to see that this pair of
equations is equivalent to $ n < kx \leq n+1$, which implies $n
\in X$. We have $s_{\alpha, 0}(n)=0$ if and only if there exists
an integer $k$ such that $\alpha (n+1) < k+1$ and $ \alpha n \geq
k$. It is not difficult to see that this pair of equations is
equivalent to $ n \leq (n-k) y < n+1$, which implies $n \in Y$.}
\end{remark}

\begin{remark}\rm{We do not know if the above results on Sturmian partitions extend to the broader class of Arnoux-Rauzy words. In fact, our proof of Lemma~\ref{key} relies on the geometric interpretation of Sturmian words as codings of orbits under an irrational rotation on the circle. It was shown in \cite{CFZ} that there exist Arnoux-Rauzy words which are not measure-theoretically conjugate to a rotation on the $n$-torus. In this case, we do not understand  which pairs of Arnoux-Rauzy words in the subshift are proximal.}
\end{remark}

\section{Proofs of Theorems~\ref{T3} \& \ref{T4}}

We begin by briefly reviewing some notions from topological dynamics. By a {\it topological flow} we mean a pair $(X,f)$ consisting of a compact set $X$ together with a homeomorphism $f$ of $X.$ In our framework we will consider $X$ to be a set consisting of bi-infinite words on a finite alphabet and $f$ the shift map.
A topological flow $(X,f)$  is said to be {\it equicontinuous} if for every $\epsilon >0,$ there exists a $\delta >0,$ such that for all $x,y\in X,$ if $d(x,y)<\delta$ then $d(f^n(x),f^n(y))<\epsilon$ for every $n\in \ints.$
A topological flow $(Y,g)$ is called a {\it factor} of $(X,f)$ if there exists a continuous surjection
\[\pi : X \rightarrow Y\]
such that $\pi \circ f=g\circ \pi.$  It is well known (for instance by way of Zorn's lemma) that every topological flow $(X,f)$ has a {\it maximal equicontinuous factor} $(Y,g)$ i.e., $(Y,g)$ is an equicontinuous factor of $(X,f)$ and any equicontinuous factor $(Z,h)$ of $(X,f)$ is also a factor of $(Y,g).$
It is also well known that if $\pi: X\rightarrow Y$ is the maximal equicontinuous factor, then for any two points $x,y\in X$ we have that $\pi(x)=\pi(y)$ if and only if $x$ and $y$ are regionally proximal (see \cite{A} ).

\begin{proof}[Proof of Theorem~\ref{T3}] Let  us fix positive integers $r$ and $N.$ Consider the constant length substitution \[\tau:\{1,2,\ldots, r\}\rightarrow \{1,2,\ldots ,r\}^+\] given by $1\mapsto 123\cdots r, $ $2\mapsto 23\cdots r1,$  $3\mapsto 34\cdots r12,$ $\ldots ,$  $r\mapsto r12\cdots r-1.$ In case $r=2$ we have the Thue-Morse substitution on the alphabet $\{1,2\}.$  For $1\leq i\leq r,$ let $x^{(i)}$ denote the $i$th fixed point of $\tau$ beginning in the letter $i.$ As in the case of Thue-Morse, for $i\neq j$ the words $x^{(i)}$ and $x^{(j)}$ never coincide, i.e., $x^{(i)}_n\neq x^{(j)}_n$ for each $n\in \nats.$  Let $(X,T)$ denote the one-sided minimal subshift generated by the primitive substitution $\tau.$  We will now show that each of the fixed points $x^{(i)}$ is distal.

\begin{lemma} Let $x$ denote any one of the fixed points $x^{(i)}$ of the substitution $\tau$ above. Then $x$ is distal. In particular, the two fixed points of the Thue-Morse substitution are each distal.
\end{lemma}

\begin{proof} Let $(\tilde X,T)$ denote the two-sided subshift generated by $\tau,$ and let
$\pi :\tilde X\rightarrow Y$ denote the maximal equicontinuous factor.  The substitution $\tau$ above is
of  Pisot type, in fact, the dilation of $\tau$ is $r$ and all other eigenvalues are equal to $0.$ (Note that $\tau$ is not an irreducible substitution). It is proved in \cite{BBK} that, for a primitive substitution of Pisot type (irreducible or not), the mapping
onto the maximal equicontinuous factor is finite to one.\footnote{The authors study the maximal equicontinuous factor of
$1$-dimensional substitutive real tiling spaces. To apply their finiteness result (Theorem 4.2 in \cite {BBK}), we use the fact that in our setting all the tiles have the same length, and hence proximality of points in $X$ with respect to the shift map $T$ implies proximality of the corresponding tilings under the $\reals-$action.}  Thus there exists a constant $C$ such that
for any $z\in \tilde X,$ there are at most $C$ points $z'\in \tilde X$ which are regionally proximal to $z$ In particular,  for any $z\in \tilde X,$ there are at most $C$ points $z'\in \tilde X$ which are proximal to $z.$

Now suppose $y\in X$ is proximal to $x. $ We will show that $y=x.$ It is easy to see that the bi-infinite word $z=x_{\mbox{rev}}\cdot x\in \tilde X$ where $x_{\mbox{rev}}$ denotes the reversal or mirror image of $x,$ and where  $\cdot$ denotes the origin. Similarly, let $y'$ denote a left infinite word such that the concatenation
 $z'=y'\cdot y\in \tilde X.$   Since $x$ and $y$ are proximal, it follows that $z$ and $z'$ are proximal.
Set $\sigma =\tau^r.$  Since $\tau, $ and hence $\sigma ,$ are of constant length,
it follows that $\sigma (z')$ is proximal to $\sigma(z).$ But  $\sigma(z)=z.$ Hence  $(\sigma ^n(z'))_{n\geq 0}$ defines an infinite sequence of points in $\tilde X$ each of which is proximal to $z,$ and which in the limit  tends to $x^{(i)}_{\mbox{rev}}\cdot x^{(j)}$ where $i$ is the first (meaning rightmost) letter of $y'$ and $j$ is the first letter of $y.$ But since there are only finitely many points in $\tilde X$ which are proximal to $z$ it follows that $\sigma ^n(z')= x^{(i)}_{\mbox{rev}}\cdot x^{(j)}$ for some $n\geq 0.$ Hence by de-substituting we obtain  $z'=x^{(i)}_{\mbox{rev}}\cdot x^{(j)}$ from which it follows that $y=x^{(j)}.$ Thus both $x$ and $y$ are fixed points of $\tau$ which are assumed proximal. It follows that $y=x$ and hence $x$ is distal as required.
\end{proof}

Put $x=x^{(1)}.$ Since $x$ is distal,  so is $T^n(x)$ for each $n\geq 1.$  On the other hand, it is easy to see that for each positive integer $n$  we have  $u^{(i)}[n]x\in X,$ where $u^{(i)}[n]$ denotes the reversal of the prefix of $x^{(i)}$ of length $n.$  Thus the $r$ words $\{u^{(1)}[n]x, u^{(2)}[n]x, \ldots, u^{(r)}[n]x\}$ are pairwise proximal and each begin in distinct letters (this is because the fixed points never coincide).
Finally let $\omega =  u^{(1)}[N+1]x,$ and set $A_i=\omega \big|_i$ for each $1\leq i\leq r.$ Then each $A_i$ is a central set. For each $1\leq n\leq N,$ we have that $A_i-n=T^n(\omega)\big|_i=u^{(1)}[N+1-n]x\big|_i$ is a central set. But  for $k\geq 1,$ we have that $A_i-(N+k)=T^{k-1}(x)\big|_i$ which is a central set if and only if $T^{k-1}(x)$ begins in $i.$

\end{proof}

\begin{proof}[Proof of Theorem~\ref{T4}] Fix a positive integer $r.$ Let $\tau$ be a primitive substitution whose associated subshift $\Omega$ is topologically weak mixing. For instance we may take the substitution $0\mapsto 001$ and $1\mapsto 11001$ or $0\mapsto 001$ and $1\mapsto 11100$ (see \cite{DK}).
Let $\omega \in \Omega.$ Fix $m$ such that $\p_{\omega}(m)\geq r,$ and put $s=\p_{\omega}(m).$
Let $u_1,u_2,\ldots ,u_s$ denote the factors of $\omega$ of length $m.$
As pointed out to us by V. Bergelson and Y. Son \cite{BergCom},  the weak mixing implies that the set of points in $\Omega$ proximal to $\omega$ is dense in $\Omega$
(see for instance page 184 of \cite{F}).
Thus for each factor $u_i$ there exists a word $x_i\in \Omega$ beginning in $u_i$ and which is proximal to $\omega.$ Hence by Theorem~\ref{Berg} there exists a minimal  idempotent ultrafilter $p_i\in \beta \nats$ such that $p_i^*(\omega)=x_i.$ Hence for each $1\leq i\leq s$ we have that $\omega \big|_{u_i}\in p_i$ and hence  $\omega \big|_{u_i}$  is a central set. Finally, for each positive integer $n$ and for each $1\leq i\leq s$ we have that
\[ \omega \big|_{u_i}-n=T^n(\omega)\big|_{u_i}.\]
Again the weak mixing implies that there exists a word $x\in \Omega$ beginning in $u_i$ and proximal to $T^n(\omega).$ Hence there exists a minimal idempotent $p\in \beta \nats$ such that $p^*(T^n(\omega))=x$
from which it follows that $\omega \big|_{u_i}-n\in p$ and hence  $\omega \big|_{u_i}-n$ is a central set.
Thus we obtain a partition of $\nats$
\[\nats = \bigcup_{i=1}^s \omega \big|_{u_i}\]
into $s$-many central sets and for each positive integer $n$ and $1\leq i\leq s$ we have that $\omega \big|_{u_i}-n$ is again a central set.
 Thus, setting
\[A_i= \omega \big|_{u_i}\] for $i=1,\ldots, r-1,$ and \[A_r=\bigcup_{i=r-1}^s \omega \big|_{u_i}\]
we obtain the desired partition of $\nats.$
\end{proof}

 \section{Infinite  central partitions of $\nats$}

In this section we construct infinite partitions of $\nats$ into central sets by using words on an infinite alphabet.
Our construction makes use of the notion of \emph{iterated palindromic closure operator} (first introduced in \cite{deLuca}):

\begin{definition}\label{ipc}
The iterated palindromic operator $\psi$ is defined inductively as follows:
\begin{itemize}
\item $\psi (\varepsilon) = \varepsilon$,
\item For any word $w$ and any letter $a$, $\psi (wa) = (\psi(w)a)^{(+)}$.
\end{itemize}
We denote with $w^{(+)}$ the \emph{right palindromic closure} of the word $w$, i.e., the shortest palindrome which has $w$ as a prefix.
\end{definition}

For example, $\psi (aaba)=aabaaabaa.$
The operator $\psi$ has been extensively studied for its central role in constructing standard Sturmian and episturmian words. It follows immediately from the definition that if $u$ is a prefix of $v,$ then $\psi(u)$ is a prefix of $\psi(v).$ Thus, given an infinite word $\omega =\omega_0\omega_1\omega_2\ldots  $ on the alphabet $A$ we can define
\[\psi(\omega)=\lim_{n\rightarrow \infty} \psi(\omega_0\omega_1\omega_2\ldots  \omega_n).\]
The following lemma summarizes the properties of $\psi$ needed.

\begin{lemma}\label{lemmapsi}
Let $\Delta$ be a right infinite word over the (finite or infinite) alphabet $A$ and let $\omega = \psi(\Delta)$. Then the following statements hold:
\begin{enumerate}
\item The word $\omega$ is closed under reversal, i.e., if $v=v_1v_2\ldots v_k $ is a factor of $\omega$, then so is its mirror image $v_k\ldots v_2v_1.$
\item \label{cond2} The word $\omega$ is uniformly recurrent.
\item \label{cond3} If each letter $a\in A$ appears in $\Delta$ an infinite number of times, then for each prefix $u$ of $\omega$ and each $a\in A,$ we have $au$ is a factor of $\omega.$
\end{enumerate}
\end{lemma}
\begin{proof}
Since any factor of $\omega$ is contained in some $\psi(v)$ for a sufficiently long prefix $v$ of $\Delta$, and $\psi(v)$ is by definition a palindrome (and hence closed under reversal), the first statement is proved. The second statement is easily derived from the fact that for any finite prefix $va$ of $\Delta$ ($a$ being a letter), we have that $|\psi(va)| \leq 2|\psi(v)|+1$ and moreover $\psi(va)$ begins and ends in $\psi(v).$ It follows that any factor of length (for example) $3|\psi(v)|$ contains an occurrence of $\psi(v)$.

Finally suppose each $a\in A$ appears infinitely many times in $\Delta.$ Thus for any letter $a$ and any prefix $v$ of $\Delta$ there exists a prefix of $\Delta$ of the form $vv'a$. From the definition of $\psi$ we then have that $\psi(vv')a$ is a prefix of $\omega$ and $\psi(vv')$ ends in $\psi(v)$, so $\psi(v)a$ is a factor of $\omega$. Since $\psi(v)$ is a palindrome and $\omega$ is closed under reversal, we obtain that for any prefix $v$ of $\Delta$ and for any letter $a$, the word $a\psi(v)$ is a factor of $\omega$ and the third statement easily follows.
\end{proof}

With the preceding Lemma, we are now able to construct infinite partitions of $\nats$ such that each element of the partition is an IP-set.

\begin{proposition}\label{wbonacci}
Let $\omega=\psi (\Delta)$ where $\Delta$ is a right infinite word on an infinite alphabet $\mathcal{A}$ with the property that each letter $a\in \mathcal{A}$ occurs in $\Delta$ an infinite number of times. Then, for any $a \in  \mathcal{A}$, the set $a\omega\big|_{a}$ is a central set, thus $\{ \omega\big|_{a}+1 \}_{a \in \mathcal{A}}$ is an infinite partition of $\nats-\{0\}$ into central sets\footnote{This is a special case of a more general result of Hindman, Leader and Strauss \cite{HLS}  in which they show that every central set in $\nats$  is a countable union of pairwise disjoint central sets.}.
\end{proposition}
\begin{proof}
From \ref{lemmapsi} we have that $\omega$ is uniformly recurrent and closed under reversal. Furthermore, since each $a\in  \mathcal{A}$ occurs in $\Delta$ an infinite number of times,   (\ref{cond3}) of  \ref{lemmapsi} implies that the set of factors of $a\omega$ coincides with that of $\omega.$
It follows therefore that $a\omega$ is uniformly recurrent as well.
Let us denote by $\pi_{a}$ the image of $\omega$ under the morphism $\mu_{a}$ defined as follows:
\begin{itemize}
\item $\mu_{a} (a) = 0$,
\item $\mu_{a} (x) = 1$ if $x \neq a$.
\end{itemize}
Since $a\omega$ is uniformly recurrent for any $a$, it is clear that also $0\pi_{a}$ is uniformly recurrent for any $a$.
From Theorem~\ref{Berg}, we then have that for any $a$ there exists a minimal idempotent ultrafilter $p_{a}$ such that $p_{a}^{*}(0\pi_{a})=0\pi_{a}$. In particular, this means, by Lemma \ref{IP}, that $0\pi_{a}\big|_{0}$ (which clearly coincides with $a\omega\big|_{a}$ by definition) is a central set for any $a$. The statement can then be easily derived from the fact that $a\omega\big|_{a}-1=\omega\big|_{a}.$
\end{proof}


\begin{thebibliography}{50}

\bibitem{ArRa}
P. Arnoux and G. Rauzy,
\newblock {\it Repr\'esentation g\'eom\'etrique de suites de complexit\'e $2n+1$,}
\newblock  Bull. Soc. Math. France {\bf 119} (1991),
199--215.

\bibitem{A}
J. Auslander,
\newblock{\it Minimal flows and their extensions,}
\newblock North-Holland Mathematical Studies, vol 153, North-Holland 1988.

\bibitem{BBK}
V. Baker, M. Barge and J. Kwapisz,
\newblock {\it Geometric realization and coincidence for reducible non-unimodular Pisot tiling spaces with an application to $\beta$-shifts,}
\newblock Num\'eration, pavages, substitutions,
\newblock  Ann. Inst. Fourier (Grenoble) {\bf 56} No. 7 (2006), p. 2213--2248.



\bibitem{VB2}
V. Bergelson,
\newblock  {\it Minimal idempotents and ergodic Ramsey theory,}
\newblock Topics in dynamics and ergodic theory, 8Ð 39,
\newblock  London Math. Soc. Lecture Note Ser., {\bf 310},
\newblock Cambridge Univ. Press, Cambridge, 2003.

\bibitem{BH}
V. Bergelson and  N. Hindman
\newblock{Nonmetrizable topological dynamics and Ramsey Theory}
\newblock Trans. Amer. Math. Soc. {\bf 320} (1990), p. 293--320.


\bibitem{BHS}
V. Bergelson, N. Hindman and D. Strauss,
\newblock  {\it Strongly central sets and sets of polynomial returns mod $1,$}
\newblock Proc. Amer. Math. Soc.,
\newblock to appear.

\bibitem{BergCom}
V. Bergelson and Y. Son
\newblock Personal communication.

\bibitem{Bl}
A. Blass,
\newblock {\it Ultrafilters: where topological dynamics = algebra = combinatorics,}
\newblock Topology Proc. {\bf 18} (1993), p. 33--56.

\bibitem{BHPZ}
M. Bucci, N. Hindman, S. Puzynina and L.Q. Zamboni
\newblock {\it On additive properties of sets defined by the Thue-Morse word}
\newblock preprint 2012.


\bibitem{CFZ}
J. Cassaigne, S. Ferenczi, and L.Q. Zamboni,
\newblock {\it Imbalances in {A}rnoux-{R}auzy sequences,}
\newblock  Ann. Inst. Fourier (Grenoble), {\bf 50} (2000), no. 4 p. 1265--1276.


\bibitem{DHS}
D. De, N. Hindman and D. Strauss,
\newblock {\it A New and Stronger Central Sets Theorem,}
\newblock Fundamenta Mathematicae {\bf 199} (2008), p. 155--175.


\bibitem{Dek}
F.M. Dekking
\newblock {\it The spectrum of dynamical systems arising from substitutions of constant length,}
\newblock Z. Wahrscheinlichkeitstheorie und Verw. Gebiete, {\bf 41} (1977/1978), p. 221--239.


\bibitem{DK}
F.M. Dekking and M. Keane
\newblock {\it Mixing properties of substitutions,}
\newblock Z. Wahrscheinlichkeitstheorie und Verw. Gebiete, {\bf 42} (1978), p. 23--33.

\bibitem{deLuca}
A. de Luca,
\newblock {\it Sturmian words: structure, combinatorics, and their arithmetics,}
\newblock Theoret. Comput. Sci. {\bf 183} (1997), p. 45--82.

\bibitem{DT1}
J.-M. Dumont and A. Thomas,
\newblock  {\it Syst\`emes de num\'eration et fonctions fractales relatifs
aux substitutions,}
\newblock Theoret. Comput. Sci., {\bf 65} (2) (1989), p. 153--169.

\bibitem{DT2}
J.-M. Dumont and A. Thomas,
\newblock  {\it Digital sum moments and substitutions,}
\newblock Acta Arith., {\bf 64}  (1993), p. 205--225.

\bibitem{EZ}
M. Edson and L.Q. Zamboni,
\newblock {\it On the number of partitions of an integer in the $m$-bonacci base,}
\newblock Num\'eration, pavages, substitutions.
\newblock Ann. Inst. Fourier (Grenoble) {\bf 56} (2006), no. 7, p. 2271--2283.

\bibitem{E}
R. Ellis,
\newblock{ \it Distal transformation groups}
\newblock{Pac. J. Math.} {\bf 8} (1958), p. 401--405.

\bibitem{FA}
D. G. Fon-Der-Flaass and  A. E. Frid,
\newblock {\it On periodicity and low complexity of infinite permutations,}
\newblock European J. of Combin. {\bf 28} (2007), p. 2106--2114.

\bibitem{F}
H. Furstenberg,
\newblock {\it Recurrence in Ergodic Theory and Combinatorial Number Theory,}
\newblock Princeton University Press, 1981.

\bibitem{H}
N. Hindman,
\newblock {\it Finite sums of sequences within cells of a partition of $\nats,$}
\newblock J. Combinatorial Theory (Series A) {\bf 17} (1974), p. 1--11.

\bibitem{H2}
N. Hindman,
\newblock {\it Ultrafilters and Ramsey theory-an update}
\newblock Set theory and its Applications (J. Steprans \& S. Watson eds.)
\newblock Lecture Notes in Mathematics {\bf 1401}, Springer-Verlag, 1989, p. 97--118.

\bibitem{HLS}
N. Hindman, I. Leader and D. Strauss,
\newblock {\it Infinite partition regular matrices: solutions in central sets,}
\newblock Trans. Amer. Math. Soc. {\bf 355} (2003), p. 1213--1235.



\bibitem{HS}
N. Hindman and D. Strauss,
\newblock {\it Algebra in the Stone-\v Cech compactification. Theory and applications,}
\newblock de Gruyter Expositions in Mathematics {\bf 27}
\newblock Walter de Gruyter \& Co., Berlin, 1998.

\bibitem{HS2}
N. Hindman and D. Strauss,
\newblock {\it A simple characterization of sets satisfying the Central Sets Theorem,}
\newblock New York J. Math {\bf 15} (2009), p. 405--413.

\bibitem{K1}
T. Kamae,
\newblock {\it Uniform sets and super-stationary sets over general alphabets,}
\newblock  Ergodic Theory \& Dynam. Systems
\newblock {\bf 31} (2011), p. 1445--1461.

\bibitem{K2}
T. Kamae,
\newblock {\it Behavior of various complexity functions,}
\newblock Theoret. Comput. Sci.
\newblock {\bf 420} (2012), p. 36--47.



  \bibitem{Lothaire1983book}
M.~Lothaire,
\newblock {\em Algebraic Combinatorics on Words}
\newblock Cambridge UK, Cambridge University Press, 2002.


\bibitem{MorHed1940}
M.~Morse and G.A. Hedlund.
\newblock {\it Symbolic {D}ynamics {II}: {S}turmian trajectories,}
\newblock Amer. J. Math., {\bf 62} (1) (1940), p. 1--42.





\bibitem{Zeck}
E. Zeckendorff,
\newblock {\it Repr\'esentation des nombres naturels par une somme de nombres de
Fibonacci ou de nombres de Lucas,}
\newblock  Bull. Soc. Royale Sci. Li\`ege,  {\bf 42} (1972), p.
179--182.
 \end{thebibliography}
\end{document}